\documentclass[12pt,reqno]{amsart}

\usepackage{latexsym}
\usepackage{amssymb}

\renewcommand{\Re}{{\operatorname{Re}\,}}
\renewcommand{\Im}{{\operatorname{Im}\,}}

\newcommand{\bbb}{|\!|\!|}

\renewcommand{\epsilon}{\varepsilon}
\newcommand{\dist}{{\operatorname{dist}}}

\newcommand{\sm}{\smallsetminus}
\newcommand{\szego}{Szeg\H{o} }

\newcommand{\inv}{^{-1}}
\newcommand{\kahler}{K\"ahler }
\newcommand{\sqrtn}{\sqrt{N}}

\newcommand{\wh}{\widehat}
\newcommand{\PP}{{\mathbb P}}
\newcommand{\N}{{\mathbb N}}
\newcommand{\R}{{\mathbb R}}
\newcommand{\C}{{\mathbb C}}

\newcommand{\Z}{{\mathbb Z}}

\newcommand{\CP}{\C\PP}
\renewcommand{\d}{\partial}
\newcommand{\dbar}{\bar\partial}
\newcommand{\ddbar}{\partial\dbar}

\newcommand{\E}{{\mathbf E}}

\newcommand{\half}{{\textstyle \frac 12}}
\newcommand{\vol}{{\operatorname{Vol}}}

\newcommand{\SU}{{\operatorname{SU}}}
\newcommand{\FS}{{{\operatorname{FS}}}}

\renewcommand{\phi}{\varphi}

\newcommand{\ccal}{\mathcal{C}}
\newcommand{\dcal}{\mathcal{D}}

\newcommand{\hcal}{\mathcal{H}}

\newcommand{\lcal}{\mathcal{L}}
\newcommand{\mcal}{\mathcal{M}}
\newcommand{\ncal}{\mathcal{N}}
\newcommand{\ocal}{\mathcal{O}}

\newcommand{\al}{\alpha}

\newcommand{\ga}{\gamma}
\newcommand{\Ga}{\Gamma}

\newcommand{\la}{\lambda}
\newcommand{\ep}{\varepsilon}
\newcommand{\de}{\delta}
\newcommand{\De}{\Delta}
\newcommand{\om}{\omega}

\newtheorem{theo}{{\sc Theorem}}[section]

\newtheorem{cor}[theo]{{\sc Corollary}}

\newtheorem{lem}[theo]{{\sc Lemma}}
\newtheorem{prop}[theo]{{\sc Proposition}}

\title[Overcrowding and hole probabilities for random zeros]
{Overcrowding and hole probabilities for random zeros on complex manifolds}

\author{Bernard Shiffman}
\address{Department of Mathematics, Johns Hopkins University, Baltimore,
MD 21218, USA} \email{shiffman@math.jhu.edu}

\author{Steve Zelditch}
\address{Department of Mathematics, Johns Hopkins University, Baltimore,
MD 21218, USA} \email{zelditch@math.jhu.edu}

\author{Scott Zrebiec}
\address{Department of Mathematics, Texas A \& M  University, College Station,
TX 77843, USA} \email{zrebiec@math.tamu.edu}

\thanks{Research of the first author partially supported by NSF grant 
DMS-0600982;
research of the second author partially supported by NSF grant DMS-0603850.}

\date{June 4, 2008}

\begin{document}

\begin{abstract} We give asymptotic large deviations estimates for the volume
inside a domain $U$ of  the zero set of a random polynomial of degree $N$, or 
more generally, of a random holomorphic section of the
$N$-th power of a positive line bundle on a compact K\"ahler manifold. In
particular, we show that for all  $\delta>0$, the probability that this volume
differs by more than $\delta N$  from its average value is less than
$\exp(-C_{\delta,U}N^{m+1})$, for some constant
$C_{\delta,U}>0$.   As a consequence, the ``hole probability" that a random
section does not vanish in $U$ has an upper bound of the form
$\exp(-C_{U}N^{m+1})$.
\end{abstract}

\maketitle

 \section{Introduction}

The purpose of this article is to prove large deviations estimates for 
probabilities of
overcrowding and undercrowding of zeros of random holomorphic sections $s_N \in 
H^0(M, L^N)$
of high powers of a positive Hermitian  line bundle $L^N \to M$  over a compact 
\kahler manifold. A special case
is that of $\SU(m + 1)$ polynomials of degree $N$.
Our main results give  rapid exponential decay rates as the degree $N \to \infty
$   for the probability that
the zero set of a
random holomorphic section of $L^N$ is too large or too small in an arbitrary 
fixed domain,  and in
particular for the ``hole probability" that it misses the domain entirely.

To state our results we need some notation; we follow \cite{SZvar} and  review 
the
relevant notation and background in \S \ref{background}.
Let $(L, h)$ be a Hermitian holomorphic line bundle  with
positive curvature
$\Theta_h$ over an $m$-dimensional compact complex manifold $M$. Then
$\om_h:=\frac  i2\Theta_h$ is a \kahler form, which induces inner products (see
(\ref{induced}))
 and associated Gaussian probability measures  $\ga_N$ (see (\ref{gaussian}))   
on
the spaces $H^0(M, L^N)$ of holomorphic
sections of powers of $L$. 

We denote the zero set of a section $s_N \in H^0(M, L^N)$ by $Z_
{s_N} = \{z: s_N(z) = 0\}$. It is a complex $(m-1)$-dimensional
hypersurface whose $(2m-2)$-dimensional volume
 in an open set $U$ is given by
\begin{equation}  \vol_{2m-2}(Z_{s_N}\cap U) = \int_{Z_{s_N}\cap U}
\frac{\omega_h^{m-1}}{(m-1)!}\,.
\end{equation}
\begin{theo}\label{volumes}  Let $(L,h)\to M$ be a positive Hermitian line 
bundle
over a compact \kahler manifold $M$ of dimension $m$,  and give $M$ the metric
with \kahler form
$\om_h=\frac i2 \Theta_h$. Let  $U$ be an open subset of $M$ such that $\d U$ 
has
zero measure in $M$.
Then for all
$\de>0$, there is a constant $C_{\de,U}>0$  such
that
$$\ga_N
\left\{ s_N:\, \left|\frac 1N\, \vol_{2m-2}(Z_{s_N}\cap U) - \frac m\pi
\,{\vol_{2m}(U)}\right| >
\delta \right\}\leq e^{- C_{\delta,U} N^{m+1}}\quad \forall\ N\gg 0\,.$$
\end{theo}

Here, $N\gg 0$ means that $N\ge N_0$ for some $N_0=N_0(\de)\in \Z^+$.
In particular, for the case where $\dim M=1$, the volume $\vol_{2m-2}(Z_{s_N}
\cap
U)$ becomes the
number $\ncal^U_N(s_N)$ of zeros of $s_N$ in $U$, and we have

\begin{cor}\label{dim1}

 Let $(L,h)\to M$ be a positive Hermitian line bundle over a compact Riemann
surface $M$, and give $M$ the metric with \kahler form
$\om_h=\frac i2 \Theta_h$.  Let
$U\subset M$ be an open set in $M$ such that $\d U$ has zero measure in
$M$.  Then for all $\de>0$, there is a constant $C_{\de,U}>0$ such that
$$ \ga_N \left\{ s_N:\;\left|\frac 1N \, \ncal^U_N(s_N) - \frac 
{\operatorname{Area}(U)}\pi \right|>
\delta\right\} \leq e^{- C_{ \de,U} N^2}\quad \forall\ N\gg 0\,.$$
\end{cor}

In the case where $M =\CP^m, L =\ocal(1)$ and $h$ is the Fubini-Study metric, the 
Gaussian ensembles $(H^0(M,L^N),\ga_N)$
coincide with the well-known $\SU(m +1)$ 
ensembles of degree $N$  polynomials,  
$$f_N=\sum_{|J|\le N} c_J 
{N\choose J}^{1/2} z_1^{j_1}\cdots z_m^{j_m}\;,$$
where
$J=(j_1,\dots,j_m)\in \N^m$, $z^J= z_1^{j_1}\cdots z_m^{j_m}$, ${N\choose J} = 
\frac{N!}{(N-|J|)!j_1!\cdots j_m!}$ and the $c_J$ are independent identically 
distributed complex Gaussian random
variables (see \cite{BSZ, SZ}).  Applying Theorem \ref{volumes} to this case, we 
obtain the same estimate for large deviations of the Fubini-Study volume of 
$Z_{f_N}\cap U$.  We also have a similar estimate for the Euclidean 
volume:

\begin{cor}\label{poly} Let $f_N$
be a degree $N$ Gaussian random $\SU(m + 1)$ polynomial, and let  
$U$ be a bounded domain in $\C^m$ such that $\d U$ 
has Lebesgue measure zero.
 Then for all
$\de>0$, there is a constant $C_{\de,U}>0$  such
that for $N$ sufficiently large, we have
\begin{equation}\label{poly1}\operatorname{Prob}
\left\{ \left|\frac 1N \vol_{2m-2}^E(Z_{f_N}\cap U) -V_U\right| >
\delta \right\}\leq e^{- C_{\delta,U} N^{m+1}}\,,\end{equation}
where $\vol^{E}$ denotes Euclidean volume in $\C^m$ and
$$V_U = \frac 1{(m-1)!}\int_U \frac i{2\pi}\ddbar\log (1+\|z\|^2)\wedge (\frac i2 
\ddbar \|z\|^2)^{m-1}\,.$$
In particular, if $U$ is the ball $B(r)$ of radius $r$ in $\C^m$, then 
\eqref{poly1} becomes
\begin{equation}\label{poly2}\operatorname{Prob}
\left\{ \left|\frac 1N \,n_{f_N}(r,0) -\frac{r^2}{1+r^2}\right|>
\delta \right\}\leq e^{- C_{\delta,B(r)} N^{m+1}}\,,\end{equation}
where $$n_f(r,0)= \frac{(m-1)!}{\pi^{m-1}r^{2m-2}}\vol^E_{2m-2}(Z_f\cap B(r))
$$ is the unintegrated Nevanlinna counting function.
\end{cor}

Corollary \ref{poly} follows by modifying the last step of the proof of 
Theorem \ref{volumes}; see \S \ref{s-poly}.

Letting $\de=\frac m\pi
\vol_{2m}(U)$ in Theorem \ref{volumes}, we obtain our estimate for the ``hole
probability":

\begin{theo} \label{hole} With the same hypotheses as Theorem \ref{volumes}, for
any
non-empty open set $U\in M$, there is a  constant $C_U>0$ such that
$$\ga_N
\{s_N:  Z_{s_N}\cap U=\emptyset\}
\leq e^{- C_{U} N^{m+1} }\quad \forall\ N\gg 0\,. $$  Furthermore, if there
exists a section in $H^0(M,L)$ that does not vanish on $\overline U$, then there
is a  constant $C'_U>0$ such that
$$\ga_N
\{s_N:  Z_{s_N}\cap U=\emptyset\}
\geq e^{- C'_{U} N^{m+1} }\quad \forall\ N\in \Z^+\,. $$ \end{theo}
The lower bound in Theorem \ref{hole} is elementary; see \S \ref{s-lower}.

Before sketching the novel features  of the proof, let us compare these results 
to prior results on
numbers (or volumes) of zeros of random analytic functions of various kinds. 
Among the  earliest results were those
of Offord \cite{O}  on excesses or deficiencies of zeros
of random entire analytic functions   \begin{equation} \label{f} f(z) = \sum_n 
a_n z^n \end{equation} in disks $D_r = \{z: |z| < r\}$ of $\C$.
The Taylor coefficients $a_n$ are assumed to be  independent random variables of 
several kinds. In \cite{So}, M. Sodin used Offord's method to
prove that the hole  probability that $\ncal^{D_r} = 0$, i.e. the probability 
that a random analytic function has no zeros in  $D_r$,  decays at least at the 
rate $O(e^{- C r^2}).$
    Peres-Virag \cite{PV}  gave
 an exact formula for the probability that $\ncal^{D_r} = k$ for a certain 
special determinantal ensemble of random analytic
 functions in the unit disk, which suggested that the hole probability should decay 
faster than $e^{- C r^2}$.
  Sodin and Tsirelson \cite{ST2} then proved that  $e^{- C' r^4} \leq \mbox
{Prob}\,\{\ncal^{D_r} = 0\} \leq e^{- C r^4}$ for certain $C, C' > 0$.
Further results on undercrowding and overcrowding were then proved by Krishnapur
\cite{Kr} for entire holomorphic functions of type (\ref{f})  on $\C$, and in 
\cite{Zr} for entire
holomorphic functions on $\C^m$. These articles are based on the properties of 
the monomials $z^n$ and the power series (\ref{f}).

Our results are concerned with analogous over- (and under-) crowding and hole 
probabilities, but in the different situation
where the domain $U$ is a fixed domain in a general \kahler manifold,  and where 
it is  the family  of analytic functions $s_N \in H^0(M, L^N)$ which changes 
with $N$. The change  is controlled
by the complex Hermitian differential geometry underlying the inner products 
(\ref{induced}) and the
associated \szego (or Bergman) kernels. In this general setting, there does not 
exist a useful power series type
representation (\ref{f}) for the analytic functions.  The representation
$s_N = \sum_j c_j S_j^N$ in terms of an orthonormal basis $\{S_j^N\}$ is almost 
useless for our large deviations estimates, in contrast to the power series 
representation of entire functions (\ref{f}) on $\C$, since  we know almost 
nothing about the basis elements $S_j^N$ on a general \kahler manifold. Thus, we 
must find an alternative to the  power series methods of the prior articles 
\cite{So,ST2,PV,Zr}.   We do this in  \S \ref{s-max}, where we replace the 
orthonormal basis $\{S^N_j\}$ by the asymptotically orthonormal
coherent states $\Phi^{z_{\nu}^N}_N$ centered on the points of  a ``lattice" 
$\{z_{\nu}^N\}$  of mesh $\frac{1}{\sqrt{N}}$ (see  (\ref{lattice})).  We then 
rely on our knowledge of the Bergman (or Szeg\H o) kernel for the inner product 
(\ref{induced}), in particular its off-diagonal asymptotics from \cite
{BSZ,SZsym},  to prove that inner products with these coherent states define  
asymptotically almost independent random variables, or equivalently,  the
values of $s_N$  at the  points $z_{\nu}^N$ are almost independent (see \S \ref
{LBEST}). We use this coherent state  analysis to prove a large deviations 
result for the maximum modulus of $s_N$ (Theorem \ref{max}); we expect it will
have other applications in complex geometry.

In an  earlier posting \cite{Zr2} (which this article supercedes), one of the 
authors  studied the same problems
for $\SU(m + 1)$ polynomials and obtained \eqref{poly2}.  In the case of 
$\SU(m+1)$ polynomials, the monomials ${N\choose J}^{1/2}z^J$ 
form an orthonormal basis and one can use
power series methods.   But  there is nothing
special about $\SU(m + 1)$
ensembles in terms of hole probabilities,  and the coherent state (i.e.\ \szego 
kernel) analysis in the present article allows
for the generalization from polynomials to sections of all positive 
holomorphic line  bundles over
\kahler manifolds.

 As in the model case of $\SU(m+1)$  polynomials of degree $N$
on $\C^{m + 1}$, the degree $N$ measures the complexity of the analytic 
functions
(or sections) $s_N$. As $N \to \infty$, the
zero set $Z_{s_N}$ of a random $s_N$ becomes denser and denser, and the 
probability
that it omits an open set $U$ becomes a very
rare event.
To be more precise, the random $Z_{s_N}$ not only becomes denser,
 but in fact  the mean random zero set $Z_{s_N}$ tends in the sense of
currents to the  curvature
$(1,1)$-form of the line bundle; i.e.,
\begin{equation}\label{mean}\frac 1 {N}\,\E_N ( [Z_{s_N}], \phi ) =\frac 1\pi
\int_M  \omega_h \wedge\phi\ +\ O\left(\frac 1{N^2}\right)\,,\qquad
\phi\in\dcal^{m-1,m-1}(M)\;,\end{equation} 
(see \cite{SZ}),
where $\E_N$ denotes the expectation of a
random variable in the ensemble\break
$(H^0(M, L^N), \ga_N)$, and $[Z]$ denotes the current of
integration over a hypersurface $Z$. 

We also have a
large deviations estimate for the ``linear statistics"
$( [Z_{s_N}], \phi )=
\int_{Z_{s_N}}\phi$ of equation \eqref{mean}: 
 \begin{theo} \label{main} Let $(L,h)\to (M,\om_h)$ be as in Theorem \ref
{volumes}, and give $M$ the metric with
\kahler form
$\om_h=\frac i2 \Theta_h$. Let
$\phi\in\dcal^{m-1,m-1}(M)$ be a smooth test form.  Then for all $\delta > 0$,
there exists $C_{\delta,\phi} > 0$ such that
$$\ga_N \left\{ s_N\in H^0(M,L^N):\; \left| \frac 1N\int_{Z_{s_N}}\phi - \frac 1\pi
\int_M \om_h\wedge
\phi
\right| >
\delta \right\}\leq e^{- C_{\delta,\phi} N^{m+1}}\qquad \forall\ N\gg 0\,.$$
\end{theo}

Theorem \ref{volumes} follows from  Theorem \ref{main}; see \S 
\ref{PROOFMAIN}. To prove Theorem \ref
{main}, we use Theorem \ref{max} on the large deviations of the maximum 
modulus,  together with an adaptation of the methods of \cite{ST2, Zr},
to obtain a large deviations estimate for the $L^1$-norm of $\log|s_N|$:
\begin{lem} \label{both} For all $\de>0$, there is a positive constant
$C_\de$ such that
\begin{equation*} \ga_N\left\{\int_M
\big|\log|s_N|_{h^N}\big|\,\ge\de N
\right\}\leq e^{-C_\de  N^{m+1}}\qquad \forall\ N\gg 0\,.\end{equation*}
 Here the integral is with respect to volume measure on $M$.
\end{lem}
The relevance of this lemma to probability distributions of zero sets 
is clear from the First
Main Theorem of value distribution theory, which says that the growth of a
zero set can
be controlled by the growth rates of the maximum modulus and the proximity
$m_f(r,0)$ to zero. This relation was used in
\cite{ST2} to obtain hole probabilities for random entire functions from
large deviations estimates on maximum
moduli and on $m_f(r,0)$,  and was then adapted
in \cite{Zr} to holomorphic
functions on $\C^m$. A key step is to show that the spherical integrals of $
\log^-|s_N|$
are bounded by
$\de N$  for all sections $s_N\in H^0(M,L^N)$ outside a set of measure at most $
e^{-C_\de N^{m+1}}$ (Lemma \ref{log-}).  Theorem \ref{main} then follows 
immediately by an application of  the
Poincar\'e-Lelong formula.

We end the introduction by noting two natural questions for further work  in 
this area. The first
is  whether there exists  an exact asymptotic decay rate for
the hole probability in Theorem \ref{hole}. Secondly, we are studying the zeros 
of one holomorphic
section and obtain   large deviation estimates for the hypersurface volumes of 
the random complex hypersurfaces
$Z_{s_N} =\{s_N = 0\}$ in open sets $U$. It would be
interesting to obtain similar results  for the point process of simultaneous 
zeros of $m$
independent sections.

\section{Background}\label{background} We review in this section the definition
of our probability measures and background on the \szego kernel from \cite
{SZvar}.

Throughout this paper,   $(L,h)$ will be a positive Hermitian holomorphic line
bundle over a compact \kahler manifold $M$ of dimension $m$.  We let $e_L$ 
denote
a nonvanishing
local holomorphic section over an open set $\Omega\subset M$. The curvature form
of $(L,h)$ is given locally over $\Omega$ by
$$\Theta_h= -\ddbar \log |e_L|_h^2\;.$$ Positivity of $(L,h)$ means that the
curvature $\Theta_h$ is positive, and we give $M$ the \kahler form $\om_h:=\frac
i2 \Theta_h$.  (The Chern form of $L$ is given by $c_1(L,h)=\frac 1\pi \om_h=
\frac
i{2 \pi}
\Theta_h$.)  The Hermitian metric $h$ on $L$ induces Hermitian metrics
$h^N$ on the powers $L^N$ of the line bundle.

We give the space $H^0(M,L^N)$ of global holomorphic sections of $L^N$ the
Hermitian inner product
\begin{equation}\label{induced}\langle s_N, s'_N\rangle = \int_M h^N(s_N,
\overline{s'_N}) \,\frac 1{m!}\om^m\;,\qquad s_N, s'_N \in
H^0(M,L^N)\,,\end{equation} induced by the metrics
$h,\om$. This inner product in turn induces the Gaussian probability measure on
$H^0(M, L^N)$,
\begin{equation}\label{gaussian}d\gamma_N(s_N):=\frac{1}{\pi^m}e^
{-|c|^2}dc\,,\quad s_N=\sum_{j=1}^{d_N}c_jS^N_j\,,\quad
c=(c_1,\dots,c_{d_N})\in\C^{d_N}\,,\end{equation} where
$\{S_1^N,\dots,S_{d_N}^N\}$ is an orthonormal basis for
$H^0(M,L^N)$, and $dc$ denotes
$2d_N$-dimensional Lebesgue measure.   The measure
$\ga_N$ is called the {\it Hermitian Gaussian measure\/} on $H^0(M,L^N)$ and is
characterized by the property that the $2d_N$ real variables $\Re c_j, \Im c_j$
($j=1,\dots,d_N$) are independent Gaussian random variables with mean 0 and
variance
$1/2$; equivalently,
$$\E_N c_j = 0,\quad \E_N c_j c_k = 0,\quad  \E_N c_j \bar c_k =
\de_{jk}\,,$$ where $\E_N$ denotes the expectation with respect to the measure $
\ga_N$.

As in \cite{Z,BSZ, SZvar},  we lift sections $s_N\in H^0(M,L^N)$ to the circle
bundle
$X{\buildrel {\pi}\over \to} M$  of unit vectors in the dual bundle $L\inv\to M$
with respect to the dual metric $h\inv$. (Since $(L,h)$ is positive, $X$ is a
strictly pseudoconvex CR manifold.)  The lift $\hat s_N:X\to \C$ of the section
$s_N$ is given by
$$\hat{s}_N(\lambda) = \left( \lambda^{\otimes N}, s_N(z)
\right)\,,\quad\la\in \pi\inv(z)\,.$$  The sections $\hat s_N$ span the space
$\hcal^2_N(X)$ of $CR$ holomorphic functions on $X$ satisfying $\hat
s(e^{i\theta}x)= e^{iN\theta}\hat s(x)$. The {\it \szego projector\/} is the
orthogonal projector
$\Pi_N:\lcal^2(X)\to\hcal^2_N(X)$, which is given by the {\it
\szego kernel}
$$\Pi_N(x,y)=\sum_{j=1}^{d_N} \wh S^N_j(x)\overline{\wh S^N_j(y)}
\qquad (x,y\in X)\;.$$ (The lifts $\wh S^N_j$ of the orthonormal sections $S^N_j
$
form an orthonormal basis for $\hcal^2_N(X)$.) The \szego kernel is also known 
as
the ``two point function," since
\begin{equation}\label{2point} \E_N\left(\hat s_N(x)\overline{\hat s_N(y)}
\right)
= \sum_{j,k} \E_N(c_j \bar c_k) \wh S_j^N(x)\overline{\wh S_k^N(y)}
=\Pi_N(x,y)\,.\end{equation}

It was shown in \cite{Ca,Z} that the \szego kernel on the diagonal has the
asymptotics:
\begin{equation}\label{Ze}\Pi_N(x,x) =  \frac {N^m}{\pi^m}
+O(N^{m-1})\,.\end{equation} We shall apply the following form of the leading
part of the off-diagonal asymptotics of the \szego kernel from \cite
{SZsym,SZvar}. We write \begin{equation}\label{PN}P_N(z,w):=
\frac{|\Pi_N(x,y)|}{\sqrt{\Pi_N(x,x)}
\sqrt{\Pi_N(y,y)}}\,,\quad x\in\pi\inv (z), \ y\in\pi\inv(w)\,.\end{equation}

\begin{prop}\label{DPdecay} {\rm \cite{SZsym,SZvar}} Let $b>\sqrt{2k}$,
$k\ge 1$.  Then
$$ P_N(z,w) =\left\{\begin{array}{ll}  [1+o(1)] e^{-\frac
N2\,\dist(z,w)^2}\,,\qquad &\mbox{uniformly for }\ \dist(z,w)\le b\,\sqrt{\frac
{\log N}{N}} \\ O(N^{-k})\,, & \mbox{uniformly for }\ \dist(z,w)\ge b\,\sqrt
{\frac
{\log N}{N}}\end{array}\right. \;.$$
\end{prop}

Proposition \ref{DPdecay} comes from a combination of Propositions 2.6--2.7 of
\cite{SZvar}, which are immediate consequences of the off-diagonal \szego kernel
asymptotics in \cite{BSZ,SZsym}. (For a short derivation of the \szego kernel
asymptotics using local reproducing kernels, see \cite{BBS}.)

In the following section, we use Proposition
\ref{DPdecay} to give a lower bound (which holds with ``rare" exceptions) for 
the
maximum modulus.  For this argument, we need the near-diagonal
estimate of Proposition \ref{DPdecay} for distances of order
$\sqrt{\log N}/\sqrt{N}$.

\section{Large deviations of the maximum modulus}\label{s-max}

For an open set $U \subset M$, we define the random variables
$$\mcal^U_N (s_N) = \sup_{z \in  U}\left\{ |s_N(z)|_{h^N}\right\} =
\sup_{\pi\inv(U)}|\hat s_N|\,,\qquad  s_N
\in H^0(M, L^N)\,. $$  The first step in our proof of Theorems
\ref{main}--\ref{volumes} is the following estimate of the probability of large
deviations of $\log \mcal^U_N$:
\begin{theo} \label{max}For $\de>0$, we have
$$\ga_N \left\{  \left|{\log \mcal^U_N(s_N)}
\right| \geq    \de N \right\} \leq e^{- C_{\delta,U} N^{m+1}} \qquad \forall\
N\gg 0.
$$
\end{theo}

We give below separate proofs of the estimate of the probabilities that
the upper bound $\de N$ and the lower bound $-\de N$  of $\log \mcal^U_N(s_N)$ 
are
violated.

\subsection{Upper bound estimate} The easy case is the upper bound.  We must 
show
that
\begin{equation}\label{UBE}\gamma_N \left\{ \mcal^U_N (s_N) > e^{\de N}\right\} 
<
e^{- C_{\de,U} N^{m+1}}\,.
\end{equation}
 This is a large deviations event since on  average $|s_N(z)|$ has polynomial
growth.

We denote by
$\Phi_N=(S^N_1:\cdots:S^N_{d_N}): M \to \CP^{d_N-1}$ the Kodaira embedding with
respect to an orthonormal basis $\{S_j^N\}$, , where $$d_N = \dim H^0(M,L^N) =
\frac {c_1(L)^m}{m!}N^m+O(N^{m-1})\,.$$ The Kodaira map $\Phi_N$ lifts to the 
map
$$\wh \Phi_N=(\wh S^N_1,\dots,\wh S^N_{d_N}):X\to\C^{d_N}\,.$$ We recall that
\begin{equation}\label{coherentstate} \Pi_N(x,y)=\sum_{j=1}^{d_N}
\wh S^N_j(x)\overline{\wh S^N_j(y)}=\langle
\wh \Phi_N(x),\wh\Phi_N(y)\rangle\,.
\end{equation} Let $s_N=\sum_{j=1}^{d_N}c_jS^N_j$  denote a random element of
$H^0(M,L^N)$, and write
$c=(c_1,\dots,c_{d_N})$. Then,
\begin{equation}\label{sNx}\hat s_N=
\sum_{j=1}^{d_N}c_j \wh S^N_j= c\cdot\wh\Phi_N\,,\end{equation} and thus
\begin{equation}\label{c}|\hat s_N(x)|=\left| c\cdot
\wh\Phi_N(x)\right|\le \|c\|\,{\| \wh \Phi_N(x)\|}
\,.\end{equation}  Recalling \eqref{Ze}, we have
\begin{equation}\label{K}\|\wh\Phi_N(x)\|^2 =  \Pi_N(x,x) =  \frac {N^m}{\pi^m}
+O(N^{m-1})\,.\end{equation}

By \eqref{c}--\eqref{K}, we have
\begin{eqnarray*} \ga_N\{s_N:  \mcal_N^U (s_N) > e^{\de N}\} &\le &
\ga_N\left\{c\in\C^{d_N}:\|c\|\sup\|\wh\Phi_N\| > e^{\de N}
\right\}\\ &\le &
\ga_N\left\{c\in\C^{d_N}:\|c\| >CN^{-m/2}e^{\de N}\right\}
\\ &\le &
\ga_N\left\{c\in\C^{d_N}:\max|c_j| >Cd_N^{-1/2}N^{-m/2}e^{\de N}\right\}\\ &\le 
&
e^{-C^2N^{-m}e^{2\de N}}\ \le \ e^{- e^{\de N}}\,,\qquad \mbox{for }\ N\gg 0\,,
\end{eqnarray*} which gives a much better upper bound estimate than \eqref{UBE}.

\subsection{\label{LBEST} Lower bound estimate}

We now apply the \szego kernel asymptotics of Proposition \ref{DPdecay} to
prove the large deviations estimate on the lower bound:
\begin{equation}\label{LBE}\gamma_N \{ \mcal^U_N (s) < e^{-\de N} \} < e^{-
C_{\de,U} N^{m+1}}\,.\end{equation}

To verify \eqref{LBE}, we choose a point $z_0\in U$ and  a $2m$-cube
$[-t,t]^{2m}$ centered at the origin in $T_{z_0}M\equiv \R^{2m}$. We choose $t$
sufficiently small so that $\exp_{z_0}\big([-t,t]^{2m}\big)\subset U$ and
$$\half \|v-w\|\le
\dist(\exp_{z_0} (v), \exp_{z_0}(w)) \le 2 \|v-w\|\,,
\quad\mbox{for }\ v,w\in [-t,t]^{2m}\;.$$  For each $N>0$,  we consider the
lattice of $n$ points $\{z^N_\nu\}$ in $U$ given by
\begin{equation}\label{lattice} z^N_\nu = \exp_{z_0}\left( \frac a\sqrtn
\,\nu\right),\quad
\nu\in \Ga_N:=\left\{(\nu_1,\dots,\nu_{2m})\in
\Z^{2m}:|\nu_j|\le
\frac {t\,\sqrtn}{a}\right\}\,,\end{equation}
 where $a$ is to be chosen sufficiently large.  The number of points is given by
\begin{equation}\label{n} n = \left(2\left\lfloor
\frac {t\,\sqrtn}{a}\right\rfloor+1\right)^{2m} = \left( \frac
{2t}{a}\right)^{2m}N^m+ O(N^{m-1/2})\,.\end{equation}   Choosing points
$\la_\nu^N\in X$ with
$\pi(\la_\nu^N)= z_\nu^N$, we consider the complex Gaussian random variables
$$\xi_\nu= \frac {\hat s_N (\la_\nu)}{\Pi_N(\la_\nu,\la_\nu)^{1/2}} =
\frac {\left(\la_\nu^{\otimes N},
s_N(z_\nu)\right)}{\Pi_N(\la_\nu,\la_\nu)^{1/2}}\,,$$ where we omit the
superscript $N$ to simplify notation. We note that $\xi_{\nu} = \langle \hat{s}
_N,
\Phi_N^{\lambda_{\nu}^N} \rangle$
where $\Phi_N^y(x) = \frac{\Pi_N(x,y)}{ \sqrt{\Pi_N(y,y)}}$ is the coherent 
state
centered at $y$.
 Since $\Pi_N(\la_\nu,\la_\nu)^{1/2} \approx
\left(\frac N\pi \right)^{m/2}$, it suffices to show that
\begin{equation}\label{w} \gamma_N \left\{\textstyle \max_\nu |\xi_\nu| < e^{-
\de
N}
\right\} < e^{- C_{\de,U} N^{m+1}}\,, \quad \mbox{for }\ N\gg
0\,.\end{equation}

We note that $\E_N(|\xi_\nu|^2)=1$, i.e. the $\xi_\nu$ are standard complex
Gaussians.   The $\xi_\nu$ are not independent random variables; instead, we now
apply the off-diagonal asymptotics of the \szego kernel to show that they are
``almost independent" in the sense that the covariances $\E_N(\xi_\mu\bar\xi_
\nu)$
are sufficiently small (for $\mu \neq \nu$). By \eqref{2point} and \eqref{PN},
these covariances satisfy
\begin{equation}\left|\E_N\left(\xi_\mu\bar\xi_\nu\right)\right|=
\frac{|\Pi_N(z_{\mu}, z_{\nu})|}{\sqrt{\Pi_N(z_{\mu}, z_{\mu})}
\sqrt{\Pi_N(z_{\nu}, z_{\nu})}} =P_N(z_{\mu}, z_{\nu})\,.\end{equation}

We now verify \eqref{w}: Let $$\De = (\De_{\mu\nu}), \qquad
\De_{\mu\nu}=\E_N\left(\xi_\mu\bar\xi_\nu\right),\qquad  \mu,\nu\in\Ga_N$$
denote the covariance matrix. Then by Proposition
\ref{DPdecay}, for $N\gg 0$ we have
\begin{equation}\label{nearfar} |\De_{\mu\nu} |\le \left\{\begin{array}{ll}
2\,e^{-\frac N2 \,\dist(z_\mu,z_\nu)^2} \quad & \mbox{if }\ \dist(z_\mu,z_\nu)
\le b\sqrt{\frac {\log N}N} \\ O(N^{-m-1})& \mbox{if }\
\dist(z_\mu,z_\nu)
\ge b\sqrt{\frac {\log N}N}\end{array}
\right.\ ,\end{equation}
where $b=\sqrt{2m+3}$.

Inspired by  the almost independence result of  \cite[Lemma 2.3]{NSV}, we claim
that for all $\eta>0$, we can choose the constant $a$ in \eqref{lattice} such 
that
for each fixed  $\mu\in\Ga_N$,
\begin{equation}\label{almost} \sum_{\nu\neq\mu}| \De_{\mu\nu}| \le \frac
12\,,\quad \mbox{for }\ N\gg 0\,. \end{equation} (Actually, we can make the sum
smaller than any positive number.)

Proof of \eqref{almost}: In equation \eqref{almost} and in the following, we fix
$\mu$; all sums are over $\nu$ only. By
\eqref{nearfar}, we have
$$ \sum_{\nu\neq\mu}| \De_{\mu\nu}| \le   \sum_{near} +  O(N\inv) \,,$$ where
\begin{eqnarray}
\sum_{near} &: =& \sum\left\{\De_{\mu\nu}:0<\dist(z_\mu,z_\nu)
\le \textstyle b\sqrt{\frac {\log N}N}\right\}\nonumber \\ &\le &\sum\left\{
2e^{-\frac N2 \,\dist(z_\mu,z_\nu)^2}:  0<\dist(z_\mu,z_\nu)
\le \textstyle b\sqrt{\frac {\log N}N}\right\}\,.\label{near}\end{eqnarray}
Since $
\dist(z_\nu,z_\mu)>\frac a{2\sqrtn}\|\nu-\mu\|$, we have for $a\gg 1$,
\begin{eqnarray*}\sum_{near}& \le &\sum_{\nu\ne\mu} 2 e^{-a^2\|\nu-\mu\|^2/8}
\ =\  \sum_{\nu\ne 0} 2 e^{-a^2\|\nu\|^2/8}
\ \le\ \sum_{\nu \ne 0}C_m \int_{\{\|x-\nu\| \le 1/3\}}
e^{-a^2\left(\|x\|-1/3\right)^2/8}\,dx
\\ &\le & C_m\int_{\{\|x\| \ge 2/3\}}  e^{-a^2\left(\|x\|-1/3\right)^2/8}\,dx
\ \le\ C_m\int_{\{\|x\| \ge 2/3\}} e^{-a^2\|x\|^2/{32}}\,dx \ \le C_m'
e^{-a^2/72} <\frac 12
\end{eqnarray*} (where $C_m,C_m'$ are constants depending only on $m$), which
verifies
\eqref{almost}.

We consider the
$\ell^\infty$ norm on $\C^n$, $$\bbb v\bbb:=\max_\mu |v_\mu|\,,\qquad
v=(v_1,\dots,v_n)\in\C^n\,,$$ which is implicit in \eqref{w}.  Write
$\De=I+A$. We note that $\De_{\mu\mu}=\E_N(|\xi_\mu|^2)=1$ and hence the 
diagonal
entries of $A$ vanish.    By
\eqref{almost}, for $N\gg 0$ we have
$$ \Big|\sum_\nu A_{\mu\nu}v_\nu\Big| \le \sum_\nu |A_{\mu\nu}|\,\bbb v\bbb \le
\frac 12\,\bbb v\bbb\,,\qquad v\in\C^n\,,$$ and hence
$$\bbb Av\bbb = \max_\mu  \Big|\sum_\nu A_{\mu\nu}v_\nu\Big|
\le \frac 12 \,\bbb v\bbb\,,\qquad\therefore\ \bbb\De v\bbb \ge \bbb v\bbb -
\bbb Av\bbb
\ge
\frac 12 \,\bbb v\bbb\,.$$  It follows that the eigenvalues of the covariance
matrix $\De$ are bounded below by $\half$. Therefore,
$\De$ is invertible and the eigenvalues of $\De^{-1/2}$ are bounded above by
$\sqrt 2$.

We now write $\xi=(\xi_\nu)\in\C^n$, and we consider
\begin{equation}\label{zeta}\zeta=(\zeta_\nu):=
\De^{-1/2}\xi\,,\end{equation} so that the
$\zeta_\mu$ are independent standard complex Gaussian random variables.
Furthermore, we have $$\bbb\zeta\bbb =\bbb\De^{-1/2} \xi\bbb \le
\|\De^{-1/2}\xi\| \le\sqrt 2 \|\xi\| \le \sqrt{2n}
\,\bbb \xi\bbb\,.$$ Recalling \eqref{n}, we let $C=(2t/a)^{2m}$ so that
$n\approx CN^m$, and thus
\begin{equation}\label{Linfty}\max_\mu |\zeta_\mu| \le \sqrt
{2n}\,\max_\mu|\xi_\mu|
\le
\sqrt{3C} \,N^{m/2}\, \max_\mu|\xi_\mu|\;,\end{equation} for $N\gg 0$.  Writing
$\ep_N = \sqrt{3C}\,N^{m/2} e^{-\de N}$, we then have
\begin{eqnarray*} \mbox{Prob}\left\{\max |\xi_\mu| \le e^{-\de N}\right\} &=&
\frac 1{\pi^n} \int_{\max|\xi_\mu|\le e^{-\de N}} e^{-\|\zeta\|^2}
\,d\zeta\\ &\le &  \frac 1{\pi^n} \int_{\max|\zeta_\mu|\le \ep_N} e^{-\|\zeta\|
^2}
\,d\zeta\\  &\le &  \frac 1{\pi^n}
\int_{|\zeta_1| \leq \epsilon_N} \cdots
\int_{|\zeta_n| \leq \epsilon_N} d\zeta\\ &=& (\ep_N)^{2n}\ \le\ (\ep_N)^{C
N^{m}}\\&=& \exp\left(\left[-\de N +\frac m2 \log N +\log \sqrt{3C}\right]
C N^m\right) \\&\le &
\exp\left(-\half\de C N^{m+1}\right)\,,\qquad\qquad \mbox{for }\ N\gg
0\;.\end{eqnarray*}
This verifies the lower bound estimate \eqref{LBE} and
completes the proof of Theorem
\ref{max}.\qed

\section{\label{MAINRESULTS} Proof of the main results}

\subsection{Proof of Lemma \ref{both} }\label{s-integral}

In  this section we use Theorem \ref{max}  to prove Lemma \ref{both}.
We recall that
$$\log^+t = \max(\log t,0)\,,\quad \log^-t: = \log^+\frac 1t= \max(-\log
t,0)\;,$$
and we use the identity $|\log t| = \log^+t+\log^-t$ to split the 
integrand of the
lemma into two parts.  Theorem \ref{max} immediately yields the bound
\begin{equation}\label{upper} \ga_N\left\{\int_M
\log^+|s_N|_{h^N}\ge \frac \de 2 N
\right\}
\leq e^{-C_\de  N^{m+1}}\,.\end{equation}  Thus it suffices to show
that\begin{equation}
\label{lower}\ga_N\left\{\int_M \log^-|s_N|_{h^N}\ge\frac \de 2 N
\right\}
\leq e^{-C_\de  N^{m+1}}\,.\end{equation}

To prove \eqref{lower}, we shall show that the
integrals of
$\log^-|s_N|_{h^N}$ over spheres are bounded above by $
\delta N$ when $s_N$ lies outside a small set.  Let $U\subset M$ be a 
coordinate neighborhood with holomorphic coordinates 
$z=(z_1,\dots,z_m):U\approx B(4)$, where 
$B(r)=\{z\in\C^m: \|z\|<r\}$ denotes the ball of radius $r$ in $\C^m$. 
We have the following bound on the spherical integrals:

\begin{lem}\label{log-} For all $\de>0$, there exist a positive constant $C_\de$
and measurable sets $E_{N,\de}\subset H^0(M,L^N)$ such that
$\ga_N(E_{N,\de})<e^{-C_\de N^{m+1}}$ and
$$ \int_{\{\|z\|=r\}} \log^-|s_N|_{h^N}\, d\sigma_r\le\de N \qquad \mbox{for }\
s_N\in H^0(M,L^N)\sm E_{N,\de},\   r\in[1,3],\ N\gg 0\;,$$ where  $\sigma_r$
denotes the invariant probability measure on the sphere $\{\|z\|=r\}$.
\end{lem}

\begin{proof} The proof given here mostly follows the
proofs in \cite{ST2,Zr},  which  implicitly use the radial
metric
$h=e^{-r^2}$. Since our metric is not radial and since we
require the exceptional sets
$E_{N,\de}$ to be independent of $r$, we need to modify the
arguments of \cite{ST2,Zr}. For example, we shall  subdivide
the radial interval
$[1,3]$, as well as the spheres, when applying the inequality
\eqref{poisson} below.

We begin with a deterministic estimate:
Decompose the unit sphere $\d B(1)$ into a disjoint union of sets
$I_1,\dots, I_q$ of diameter
$\le \delta^{2m+2}$. (The number $q$ depends on $\de$; for an optimal
decomposition, $q \sim \de^{-(2m-1)(2m+2)}$, but this estimate for $q$ is
unimportant and any decomposition with this diameter bound will do.) Suppose
that $r\in [1,3]$ and
$\de\in(0,\half)$, and let $\zeta_k\in\C^m$ such that
\begin{equation}\label{zetak}
\dist\big(\zeta_k,(r-\de) I_k\big)<\de^{2m+2}\,,\end{equation}  for
$1\le k\le q$.
Let $u$ be a subharmonic function on the  ball $B(4)$.   The following
estimate is given in  the proofs in \cite{ST2,Zr}:
\begin{equation}\label{poisson}\int_{\{\|z\|=r\}} u(z)\,d\sigma_r(z) \ge
\sum_{k=1}^q \mu_k\,u(\zeta_k) - C_m\de \int
_{\{\|z\|=r\}}|u(z)|\,d\sigma_r(z)\,,\end{equation}  where
$\mu_k=\sigma_1(I_k)$ (so that $\sum \mu_k=1$), and $C_m$ is a constant
depending only on $m$.

For completeness, we provide a proof of \eqref{poisson} here:
Let
$P_r(\zeta,z)= r^{2m-2}
\frac{r^2-\|\zeta\|^2}{\|\zeta-z\|^{2m}}$ denote the Poisson kernel for the
$r$-ball
$B(r)$, normalized so that  $\psi(\zeta)=\int
P_r(\zeta,z)\psi(z)\,d\sigma_r(z)$ for harmonic functions $\psi$.  Since $u$ is
subharmonic, we have
\begin{eqnarray}
\sum_{k=1}^q \mu_ku(\zeta_k) &\le & \int _{\{\|z\|=r\}}\sum_{k=1}^q
\mu_k P_r(\zeta_k,z) u(z)\,d\sigma_r(z)\nonumber
\\ &\le &  \int\left|\sum_{k=1}^q
\mu_k P_r(\zeta_k,z) -1\right|\, |u(z)|\,d\sigma_r(z) +  \int
u(z)\,d\sigma_r(z)\,.\label{p2}\end{eqnarray}
Next we bound the quantity  $\sum \mu_k P_r(\zeta_k,z) -1$. By the
O$(2m)$-invariance of
$P_r(\zeta,z)$,  we have
$\int P_r(\zeta,z)\,d\sigma_s(\zeta)=1$ for
$0<s<r=\|z\|$.  Since $\|\zeta_k -\zeta\|< 4\de^{2m+2}$ for $\zeta\in
(r-\de)I_k$, we have for
$\|z\|=r$,
\begin{eqnarray*}   \left|\sum_{k=1}^q
\mu_k P_r(\zeta_k,z) -1\right| &=& \left|\sum_{k=1}^q
\mu_k P_r(\zeta_k,z) - \int_{\{\|\zeta\|=r-\de\}}
P_r(\zeta,z)\,d\sigma_{r-\de}(\zeta)\right|\\
&\le & \sum_{k=1}^q \mu_k \int_{(r-\de)I_k} |P_r(\zeta_k,z) -
P_r(\zeta,z)|\,d\sigma_{r-\de}(\zeta)\\ &\le & 4\de^{2m+2}
\sup_{\|\zeta\|\le r-\de/2} \|d_\zeta P_r(\zeta,z)\|\;.\end{eqnarray*} Since
$\|d_\zeta P_r(\zeta,z)\| \le C'_m r^{2m}
\|\zeta-z\|^{-2m-1}$, we conclude that
\begin{equation}\label{pest} \left|\sum_{k=1}^q
\mu_k P_r(\zeta_k,z) -1\right|\le C_m\de\,, \qquad \mbox{for }\
\|z\|=r\;.\end{equation} The inequality
\eqref{poisson} follows from \eqref{p2}--\eqref{pest}.

We now use the following notation:  $A(N,r,\de)  \lesssim B(N,r,\de)$
means that for  all $\de\in (0,\half)$, there
exist  a positive integer $N_0(\de)$, a positive constant $C_\de$, and sets
$E_{N,\de}\subset H^0(M,L^N)$ of  measure
$<e^{-C_{\delta} N^{m+1}}$ such that
$A(N,r,\de)
\le B(N,r,\de)$  whenever $s_N$ is not in $E_{N,\de}$, for
 $N\ge N_0(\de)$, $r\in [1,3]$. (The constants $N_0(\de)$,
$C_\de$ and exceptional sets
$E_{N,\de}$ are independent of $r$.)  We note that the relation  $\lesssim$ is
transitive; furthermore, if  $A_1  \lesssim B_1$ and
$A_2
\lesssim B_2$, then  $A_1 + A_2  \lesssim B_1
+B_2$.  We also write
$A(N,r,\de)  \gtrsim B(N,r,\de)$) when  $B(N,r,\de)  \lesssim
A(N,r,\de)$.

Let $s_N\in H^0(M,L^N)$ and write $s_N=f_N\,e_L^{\otimes N}$, where $e_L$ is a
local frame over $U$.  We claim that
\begin{equation}\label{claim} \int _{\{\|z\|=r\}} \big|\log
|f_N(z)|\,\big|\,d\sigma_r(z) \lesssim K_1N\,.\end{equation} Here, and in the
following, $K_1,K_2,K_3,K_4$ denote  constants independent of $\de,N,r$ (but
depending on $M,L,h,U$).

We let $\al(z)=\log |e_L(z)|^2_h$ so that
\begin{equation}\label{norm} \log |s_N|_{h^N} = \log |f_N|+ \frac
N2\,\al\;.\end{equation} To prove  \eqref{claim}, we   note that it follows from
the upper bound estimate
\eqref{upper} that
\begin{equation}\label{log+}\int _{\{\|z\|=r\}} \log^+
|s_N(z)|_{h^N}\,d\sigma_r(z)\le\log\mcal^U_N(s_N)\lesssim \de N\,,\end{equation}
and therefore
\begin{equation}\label{log+f}\int _{\{\|z\|=r\}} \log^+
|f_N|\,d\sigma_r\le \int _{\{\|z\|=r\}} \log^+
|s_N|_{h^N}\,d\sigma_r +\frac N2\int _{\{\|z\|=r\}}|\al|\,d\sigma_r
\lesssim  K_2N.\end{equation}

By  Theorem \ref{max},  we can choose a point
$\zeta_0\in B(1/4)$  such that
$\log|s_N(\zeta_0)|_{h^N}\ge -N$, and thus by \eqref{norm},
$\log |f_N(\zeta_0)|\ge -K_3N$, unless $s_N$  lies  in a
set of measure $\le\exp(-C_{1,B(1/4)}N^{m+1})$. By the Poisson formula,
$$\log |f_N(\zeta_0)|+ \int_{\{\|z\|=r\}} P_r(\zeta_0,z) \log ^-|
f_N(z)|\,d\sigma_r(z) \le \int_{\{\|z\|=r\}} P_r(\zeta_0,z) \log ^+|
f_N(z)|\,d\sigma_r(z) \,.$$ Let $C\in\R^+$ such that
$$C\inv \le P_r(\zeta_0,z) \le C\,,\qquad \mbox{for }\ \|\zeta_0\|\le 1/4\,,\
\|z\|=r\in [1/2,3]\,.$$  Therefore,
\begin{eqnarray*}C\inv\int_{\{\|z\|=r\}}  \log ^-| f_N(z)|\,d\sigma_r(z)& \le&
\int_{\{\|z\|=r\}} P_r(\zeta_0,z) \log ^-| f_N(z)|\,d\sigma_r(z)\\& \le & C
\int_{\{\|z\|=r\}} \log ^+| f_N(z)|\,d\sigma_r(z) -\log |f_N(\zeta_0)|\\ &
\lesssim & CK_2N+K_3N\;,\end{eqnarray*} which together with \eqref{log+f} yields
the claim \eqref{claim}.

We now construct an open covering $\{U_{kj}\}$ of the annulus $\{
1/2\le\|z\|\le 3\}$ as follows: Let
$r_j=\frac 12 +\frac 56
\de^{2m+2}j$, for
$0\le j\le p:=\lceil 3\de^{-2m-2}\rceil$. Let   $I_1,\dots,
I_q$ be disjoint sets  of diameter
$\le \delta^{2m+2}$ decomposing the unit sphere $\d B(1)\subset \C^m$, as
above. Then the open  sets
$$U_{kj}: =\{z\in
\C^m:
\dist(z,r_j I_k)< \half\de^{2m+2}\}\,,\qquad 1\le k\le q,\ 1\le j\le p\,,$$ 
cover
the annulus $\{ 1/2\le\|z\|\le 3\}$.

Next we apply Theorem \ref{max} to choose points
$\zeta_{kj}\in U_{kj}$ such that
$\log|s_N(\zeta_{kj})|_{h^N}>-\de\,N$ for all $k, j$, unless $s_N$ lies in an
exceptional set $E_{N,\de}$ of measure $$\ga_N(E_{N,\de}) \le \sum_{k=1}^q
\sum_{j=1}^p\exp({-C_{\de,U_{kj}}N^{m+1}}) \le e^{-C_\de N^{m+1}}\qquad
\mbox{for }\ N\gg 0\,.$$  Let
$r\in [1,3]$,
$\de
\in (0,\half]$ be fixed.  Choose $j$ such that
$|r-\de-r_j|<
\half\de^{2m+2}$.  Then
$$z\in rI_k\implies |z-\zeta_{kj}|\le \left|z- \frac{r_j}{r}\,z\right| +
\left|\frac{r_j}{r}\,z-
\zeta_{kj}\right| < |r-r_j| + \frac 12 \de^{2m+2} + 3  \de^{2m+2}<\de + 4
\de^{2m+2}<2\de.$$ Thus
\begin{equation}\label{elem}\int_{\{\|z\|=r\}} \al(z)\,d\sigma_r(z) = \sum_{k=1}
^q
\int_{rI_k} \al\,d\sigma_r \ge
\sum_{k=1}^q \mu_k\al(\zeta_{kj}) - 2\de \sup |d\al|\,.\end{equation}
Since $$\dist\big(\zeta_{kj},(r-\de)I_k\big)  \le
\dist\big(\zeta_{kj},r_jI_k\big) +|r-\de-r_j|  \le \half \de^{2m+2} +\half
\de^{2m+2} =  \de^{2m+2}\,,$$ we have by
\eqref{poisson},
\begin{equation}\label{poisson2}\int_{\{\|z\|=r\}} \log |f_N(z)|\,d\sigma_r(z) 
\ge
\sum_{k=1}^q \mu_k\log|f_N(\zeta_{kj})| - C_m\de \int _{\{\|z\|=r\}} \big|\log
|f_N(z)|\,\big|\,d\sigma_r(z)\,,\end{equation}
Recalling \eqref{norm}, we combine \eqref{elem}--\eqref{poisson2} to conclude
that
\begin{eqnarray*}\int_{\{\|z\|=r\}} \log |s_N|_{h^N}\,d\sigma_r &=&
\int_{\{\|z\|=r\}} \log |f_N|\,d\sigma_r + \frac N2 \int_{\{\|z\|=r\}}
\al\,d\sigma_r
\\&
\gtrsim &
\sum_{k=1}^q \mu_k\log|s_N(\zeta_{kj})|_{h^N} - C_m\de \int_{\{\|z\|=r\}}
\big|\log |f_N|\,\big|\,d\sigma_r -N\de\sup|d\al|\,,\end{eqnarray*} Thus by
\eqref{claim} and the choice of the $\zeta_{kj}$,
\begin{equation}\label{r1} \int_{\{\|z\|=r\}} -\log |s_N|_{h^N}\,d\sigma_r
\lesssim
\de N + C_m\de K_1 N +N\de\sup|d\al| =K_4\de N\,.\end{equation} Therefore by
\eqref{log+} and \eqref{r1},
\begin{eqnarray*}\int_{\{\|z\|=r\}}
\log^-\|s_N\|_{h^N}\, d\sigma_r &=&\int_{\{\|z\|=r\}}
-\log\|s_N\|_{h^N}\, d\sigma_r+
\int_{\{\|z\|=r\}}
\log^+\|s_N\|_{h^N}\, d\sigma_r\\ & \lesssim &K_4\de N +\de N
\,.\end{eqnarray*} \end{proof}

We now use Lemma \ref{log-} to verify the estimate \eqref{lower}:  Since it is
difficult to control the exceptional set for the spherical integral in the
lemma as the radius  $r\to 0$, we cover $M$ by a finite number of 
coordinate annuli of the form
$B(3)\sm B(1)$. Integrating the inequality of
Lemma \ref{log-} over $1\le r\le 3$, we conclude that
$$\int_{B(3)\sm
B(1)}\log^-|s_N|_{h^N}
\le K\de N\,, \quad
\mbox{for }\ s_N\in H^0(M,L^N)\sm E_{N,\de}\,.$$ The estimate  
\eqref{lower} follows by summing over the annuli.  This completes the proof of
Lemma~\ref{both}.\qed

\newpage
\subsection {Completion of the proofs.} 
\subsubsection{Proof of Theorem
\ref{main}.} We let $\phi\in\dcal^{m-1,m-1}(M)$ be an arbitrary test form.  
By  the Poincar\'e-Lelong formula
\begin{equation} \frac{\sqrt{-1}}{\pi } \partial
\bar{\partial}\log|s_N|_{h^N}= [Z_{s_N}]- \frac N\pi\, \om_h\;,
\label{PL}\end{equation} we have
 \begin{eqnarray} \nonumber \left| \int_{Z_{s_N}}\phi - \frac N\pi\int_M
\om_h\wedge
\phi \right| &=& \left|\left(\frac{\sqrt{-1}}{
\pi } \partial
\bar{\partial}\log|s_N|_{h^N},\,\phi\right)\right|\\ \nonumber &=& \frac 1\pi
\left|\int_M\log|s_N|_{h^N} \,\ddbar\phi\right|
\\ \nonumber&\le & \frac 1\pi
\int_{M}\big|\log|s_N|_{h^N}\big|\,|\ddbar\phi|\, d\vol_M\\ & \le& \frac
{\sup|\ddbar\phi|}\pi
\int_{M}\big|\log|s_N|_{h^N}\big|\,d\vol_M\,,\label{a}
\end{eqnarray} where
$d\vol_M =\frac 1{m!}\om_h^m$ is the volume form on $M$.
The conclusion of Theorem
\ref{main} follows by combining Lemma \ref{both} and \eqref{a}.\qed

\subsubsection{Proof of Theorem \ref{volumes}.}\label{PROOFMAIN} 
We let $\de>0$ be arbitrary, and we choose
$\psi_1,\psi_2\in\ccal^\infty_\R(M)$ such that
$$0\le \psi_1 \le \chi_U \le \psi_2\le 1,\quad \int_M\psi_1\,d\vol_M \ge
\vol_{2m}(U)-\de, \quad  \int_M\psi_2\,d\vol_M \le
\vol_{2m}(\overline U)+\de\,.$$
We now let $\phi_j=
\frac{\psi_j}{(m-1)!}\,\om_h^{m-1}$, for $j=1,2$.   For $s_N$ not in an
exceptional set of measure $<e^{-C_{\phi_2}N^{m+1}}$ (note that
$C_{\phi_2}$ depends on
$\de$ and $U$), we have by Theorem \ref{main},
\begin{eqnarray}\vol_{2m-2}(Z_{s_N}\cap U) &=& \int_{Z_{s_N}} \chi_U
\frac{\om_h^{m-1}}{(m-1)!}\
\le\ \int_{Z_{s_N}}\phi_2\ \le\ \frac N\pi \int_M \om_h\wedge \phi_2 +\de N
\nonumber\\ &=& \frac N\pi\int_M m \psi_2\,d\vol_M +\de N\ \le\ \frac
{Nm}\pi\,\vol_{2m}(\overline U) + \left( \frac m\pi +1\right)\de
N.\label{outer}\end{eqnarray} Using $\psi_1,\phi_1$, we similarly conclude that
for $s_N$ not in  an exceptional set of measure $<e^{-C_{\phi_2}N^{m+1}}$,
\begin{equation} \label{inner} \vol_{2m-2}(Z_{s_N}\cap U) \ge \frac
{Nm}\pi\,\vol_{2m}(U) - \left( \frac m\pi +1\right)\de N\,.
\end{equation}  Replacing
$\left( \frac m\pi +1\right)\de$ by $\de$ in the above, we obtain Theorem
\ref{volumes}.  (Alternatively, since $\vol(Z_{s_N})$ is constant, we can  
obtain
\eqref{inner} by applying \eqref{outer} to $M\sm \overline U$ and using the fact
that $\vol(Z_{s_N}\cap \d U)=0$ a.s.)\qed

\subsubsection{Proof of Corollary \ref{poly}}\label{s-poly}

In terms of  the affine coordinates $z\in \C^m\subset\CP^m$, the Fubini-Study 
metric $h=(1+\|z\|)^{-2}$ induces the \kahler form $\om_\FS=\frac i2 \Theta_h= 
\frac i2 \ddbar \log (1+\|z\|)^2$. We modify the argument in the proof of 
Theorem \ref{volumes} in \S \ref{PROOFMAIN}: 
 Let $\al=\frac 1{(m-1)!\pi}\,\om_\FS\wedge (\frac i2 \ddbar \|z\|^2)^{m-1}$, so 
that $V_U=\int_U\al$. Let $\de>0$ be arbitrary, and
choose $\psi_1,\psi_2\in\ccal^\infty_\R(M)$ such that
$$0\le \psi_1 \le \chi_U \le \psi_2\le 1,\quad \int_M\psi_1\al \ge
\int_U \al-\de, \quad  \int_M\psi_2\al \le
\int_{\overline U}\al+\de\,,$$ and let $\phi_j=
\frac{\psi_j}{(m-1)!}\,(\frac i2 \ddbar \|z\|^2)^{m-1}$, for $j=1,2$. By Theorem 
\ref{main}, for $f_N$ not in a set of measure $<e^{-CN^{m+1}}$, we have
$$\vol_{2m-2}^E(Z_{f_N}\cap U)  \le \int_{Z_{f_N}}\phi_2 \le N\left(\int \psi_2
\al +\de\right) \le N(V_U+2\de)\;,
$$ and similarly, $$\vol_{2m-2}^E(Z_{f_N}\cap U) \ge  N(V_U-2\de)\;,$$
which verifies \eqref{poly1}.  Evaluating the integral for $V_{B(r)}$, we obtain \eqref{poly2}.\qed

\subsubsection{Proof of Theorem \ref{hole}}\label{s-lower}  The upper bound on the
probability is an immediate consequence of Theorem \ref{volumes}, so we need 
only
show the lower bound.  Choose a section $\sigma\in H^0(M,L)$ such that
$\sup_M|\sigma|_h=1$ and $\sigma$ does not vanish on $\overline U$.  For each
$N\ge 1$, we let
$$S^N_1 = \|\sigma^{\otimes N}\|_{L^2}^{-1}\, \sigma^{\otimes N}\in
H^0(M,L^N)\;.$$  We then complete $\{S^N_1\}$ to an orthonormal basis
$\{S^N_1,S^N_2,\dots,S^N_{d_N}\}$ for  $H^0(M,L^N)$.

Since $\|\sigma^{\otimes N}\|_{L^\infty}=1$, we have
$$\|\sigma^{\otimes N}\|_{L^2} \le \vol(M)^{1/2},\quad \forall\ N>0\;.$$

Let $$s_N = c_1 S^N_1+ \cdots + c_{d_N} S^N_{d_N} = c_1 S^N_1+ s'_N$$ be a 
random
section in $H^0(M,L^N)$. Recalling \eqref{Ze}, we have
\begin{equation}\label{CS}|s'_N(z)|_{h^N}= \left|\sum_{j\ge 2} c_j
S^N_j(z)\right|_{h^N} \le \|c'\| \left(\sum_{j\ge
2}|S^N_j(z)|_{h^N}^2\right)^{1/2} \le \|c'\| \Pi_N(z,z)^{1/2} \le \|c'\|
N^{m/2},\end{equation}  for $z\in M$, $N\gg 0$, where $c'= (c_2,\dots,c_{d_N})$.
Write $$\inf_U |\sigma|_h= e^{-a}\,.$$ (Note that $a>0$ since $\sup_U
|\sigma|_h\le 1$ and $i\ddbar \log|\sigma|_h <0$.)  Therefore,
\begin{equation}\label{infU}\inf_U |S^N_1|_{h^N}=  \frac{e^{-a
N}}{\|\sigma^{\otimes N}\|_{L^2}}  \ge b\, {e^{-a
N}}\,,\quad\mbox{where }\ b=\vol(M)^{-1/2}\,.\end{equation}

Let
$$t_N= \frac {b\,e^{-a N}}{ N^{m/2}\sqrt{d_N}}\,.$$ Since $\|c'\|\le
\sqrt{d_N}\,\max_{j\ge 2} |c_j|$, it follows from
\eqref{CS}--\eqref{infU} that for $N\gg 0$, we have
\begin{multline*} \left\{s_N= \sum c_jS^N_j: |c_1|>1,\ |c_j|<t_N
\ \mbox{ for }\ j\ge 2\right\}\\ \subset \left\{s_N=c_1 S^N_1+ s'_N :
\inf_U|c_1 S^N_1|_{h^N}> b\, {e^{-a N}},\ \sup_M |s'_N|_{h^N}\le  b\, {e^{-a
N}}\right\}
\subset\left\{s_N: Z_{s_N}\cap U=\emptyset\right\}.\end{multline*} Using the
estimate Prob$\{|c_j|\le t\}\ge t^2/2$ for $t<1$, we then conclude that
$$\ga_N \{s_N: Z_{s_N}\cap U=\emptyset\} \ge e\inv (t_N^2/2)^{d_N-1}
\ge (CN^{-m}e^{-a N})^{2d_N} \ge e^{-C' N^{m+1}}\,,$$ for $N\gg 0$, where $C,C'$
are positive constants independent of $N$.\qed


\begin{thebibliography}{WWW}

\bibitem[BBS]{BBS} R. Berman, B. Berndtsson and J. Sj\"ostrand, A direct
approach to Bergman kernel asymptotics for positive line bundles,
arXiv:math/0506367v2.


\bibitem[BSZ]{BSZ} P. Bleher, B. Shiffman and S. Zelditch, Universality and
scaling of correlations between zeros on complex manifolds,  {\it Invent.\ 
Math.}
142 (2000), 351--395.

\bibitem[Ca]{Ca} D. Catlin, The Bergman kernel and a theorem of Tian, in: {\it
Analysis and Geometry in Several Complex Variables\/}, G. Komatsu and M.
Kuranishi, eds., Birkh\"auser, Boston, 1999.


\bibitem[Kr]{Kr} M. Krishnapur, Overcrowding estimates for zeroes of planar and
hyperbolic Gaussian analytic functions, {\it  J. Stat.\ Phys.}  124  (2006),
1399--1423.


\bibitem[NSV]{NSV} F. Nazarov, M. Sodin and A. Volberg,  Transportation to 
random
zeroes by the gradient flow, {\it Geom.\ Funct.\ Anal.}
 17 (2007), 887--935.

 \bibitem[Of]{O} A. C. Offord, The distribution of zeros of power series
 whose coefficients are independent random variables, {\it Indian J. Math.} 9
(1967), 175--196.

\bibitem[PV]{PV} Y. Peres and B. Vir\'ag,  Zeros of the i.i.d.\ Gaussian power
series: a conformally invariant determinantal process, {\it   Acta Math.}  194
(2005),  1--35.

\bibitem[SZ1]{SZ} B. Shiffman and S. Zelditch, Distribution of zeros of random 
and
quantum chaotic sections of positive line bundles, {\it Comm.\ Math.\ Phys.} 200
(1999), 661--683.

\bibitem[SZ2]{SZsym} B. Shiffman and S. Zelditch,  Asymptotics of almost
holomorphic sections of ample line bundles on symplectic manifolds,  {\it J. 
Reine Angew.\ Math.}  544 (2002), 181--222.


\bibitem[SZ3]{SZvar} B. Shiffman and S.  Zelditch, Number variance
 of random zeros on complex manifolds, {\it Geom.\ Funct.\ Anal.}, to appear
(arXiv:math/0608743v2).

\bibitem[So]{So} M. Sodin,  Zeros of Gaussian analytic functions, 
{\it Math.\ Res.\ Lett.} 7 (2000), 371--381.


\bibitem[ST1]{ST1} M. Sodin and B. Tsirelson, Random complex zeros. I. 
Asymptotic
normality, {\it Israel J. Math.} 144 (2004), 125--149.

\bibitem[ST2]{ST2} M. Sodin and B.  Tsirelson, Random complex zeroes, III.  
Decay
of the hole probability, {\it Israel J. Math.}  147  (2005), 371--379.


\bibitem[Ze]{Z} S. Zelditch, \szego kernels and a theorem of Tian, {\it 
Internat.\
Math.\ Res.\ Notices} 1998 (1998),  317--331.

\bibitem[Zr1]{Zr} S. Zrebiec, The zeros of flat Gaussian random holomorphic
functions on $\C^n$, and hole probability, {\it Michigan Math.\ J.} 55 (2007),
269--284.

\bibitem[Zr2]{Zr2} S. Zrebiec, The order of the  decay of the hole probability 
for
Gaussian random $\SU(m + 1)$ polynomials, arXiv:math/0704.2733v1.

\end{thebibliography}
\end{document}